\documentclass[11pt,a4paper,reqno]{amsart}
\usepackage{a4wide}
\usepackage{amssymb}
\usepackage{amsmath}
\usepackage{amsthm}
\usepackage{latexsym}
\usepackage{mathrsfs}
\usepackage{comment}
\usepackage{graphicx}
\usepackage{caption}
\usepackage{todonotes}
\usepackage[normalem]{ulem}
\usepackage{float}
\usepackage{color}

\numberwithin{equation}{section}
\newtheorem{thm}[equation]{Theorem}
\newtheorem{prop}[equation]{Proposition}
\newtheorem{lem}[equation]{Lemma}
\newtheorem{corol}[equation]{Corollary}
\theoremstyle{definition}
\newtheorem{defn}[equation]{Definition}
\theoremstyle{remark}
\newtheorem{rem}[equation]{Remark}
\newtheorem{example}[equation]{Example}

\newcommand{\R}{\mathbb{R}}
\newcommand{\eps}{\varepsilon}
\newcommand{\beq}{\begin{equation}}
\newcommand{\eeq}{\end{equation}}

\newcommand{\Hmm}[1]{\leavevmode{\marginpar{\tiny%
$\hbox to 0mm{\hspace*{-0.5mm}$\leftarrow$\hss}%
\vcenter{\vrule depth 0.1mm height 0.1mm width \the\marginparwidth}%
\hbox to
0mm{\hss$\rightarrow$\hspace*{-0.5mm}}$\\\relax\raggedright #1}}}

\begin{document}

\title[On the Phragm\'{e}n--Lindel\"{o}f and the superposition principles]{ On the Phragm\'{e}n-Lindel\"{o}f  and the superposition principles for the $p$-Laplacian
}

\author{Pier Domenico Lamberti}

\address{%
Dipartimento di Tecnica e Gestione dei Sistemi Industriali\\
Universit\`a degli Studi di Padova\\
Stradella S. Nicola 3\\
36100 Vicenza\\
Italy}

\email{lamberti@math.unipd.it}

\author{Vitaly Moroz}
\address{
Department of Mathematics\\ 
University of Swansea\\
Swansea SA1 8EN\\
Wales, UK}
\email{v.moroz@swansea.ac.uk}



\begin{abstract} 
	We study sub and supersolutions for the $p$-Laplace type elliptic equation of the form
	$$-\Delta_p u-V|u|^{p-2}u=0\quad\text{in $\Omega$},$$
	where $\Omega$ is a radially symmetric domain in $\R^N$ and $V(x)\ge 0$ is a continuous potential such that the solutions of the equation satisfy the comparison principle on bounded subdomains of $\Omega$.
	In this work we  establish a superposition principle and then use it to develop a version of a  Phragm\'{e}n--Lindel\"{o}f comparison principle in the case $p\ge 2$.  Moreover, by applying this principle to the case of Hardy-type potentials we recover and improve  a number of known lower and upper estimates for sub and supersolutions. 
\end{abstract}

\maketitle 


\section{Introduction}

\subsection{Background}
Given a sufficiently smooth domain $\Omega$ in $\mathbb{R}^N$ with $N\ge 2$, we consider the  equation 
\begin{equation}\label{e1}
	-\Delta_p u-V|u|^{p-2}u=0,\ \ {\rm in}\ \Omega ,
\end{equation}
where $\Delta_p u=\nabla\cdot\big(|\nabla u|^{p-2}\nabla u\big)$ is the $p$--Laplacian with $p>1$ and $V\in C(\Omega)$ is a potential. 
Note that in this paper by domain we mean a connected open set. 
Unless otherwise indicated, solutions and sub or supersolutions of \eqref{e1} are always understood in the classical sense, i.e. as $C^2$--functions that satisfy the relevant equation or inequality. This is not really restrictive for the type of results we prove in this paper, see Section~\ref{section2} for details. Moreover, throughout the paper we assume that 
\begin{equation*}
\text{\it equation \eqref{e1} admits a positive supersolution in $\Omega$ that is not a solution,}\leqno{(*)}
\end{equation*} 
that is, there exists $0<\phi\in C^2(\Omega)$ 
such that $-\Delta_p\phi-V\phi^{p-1}\ge 0$ and $-\Delta_p\phi-V\phi^{p-1}\neq 0$ in $\Omega$. 
Note that we do not explicitly assume any boundary regularity of the supersolution $\phi$.

If $V\le 0$ and $V\neq 0$ then a constant is always a strict supersolution of \eqref{e1} and hence $(*)$ holds automatically. If $V\geq 0$ or $V$ changes sign then $(*)$ is an implicit assumption about the size and the structure of the potential $V$. It is well known that at least under mild boundary regularity assumptions condition $(*)$ is equivalent to the validity of the maximum principle for the equation \eqref{e1} on bounded domains, see \cite[Theorems 2, 3]{Garcia-Melian} (see \cite{PT} for similar results on the unbounded domains).
\begin{thm}[\sc Maximum principle]\label{thmMP}
The following two properties are equivalent:	
\begin{itemize}
	\item[$i)$] $(*)$ holds;\smallskip
	\item[$ii)$] if $u\in C^2(\Omega)\cap C(\bar \Omega)$ and $u_{|\partial\Omega}\in C^{1,\alpha}(\partial\Omega)$ with $0<\alpha \le 1$, are such that
	\begin{equation} \label{maxp}
		\left\{ 
		\begin{array}{rcll}
			-\Delta_p u-V|u|^{p-2}u&\ge&  0&\text{in $\Omega$,}\\
			u&\ge& 0&\text{on $\partial\Omega$,}
		\end{array}
		\right.
	\end{equation} 
then $u\ge 0$ in $\Omega$.
\end{itemize}
\end{thm}

While in the linear case $p=2$ the relation between maximum and comparison principles for \eqref{e1} is straightforward and follows from additivity, this is no longer the case in the nonlinear case $p\neq 2$, unless $V\le 0$. In fact, the only situation when
$-\Delta_p u-V|u|^{p-2}u\le-\Delta_p v-V|v|^{p-2}v$ in a bounded smooth domain $\Omega$ and $u\le v$ on $\partial\Omega$ implies that $u\le v$ in $\Omega$ is the case when $V\leq 0$ in $\Omega$, i.e. when constants are supersolutions (see \cite[Section 5]{Garcia-Melian}). In general, additional restrictions are necessary for a comparison principle to be valid when $(*)$ holds. The following was proved in \cite[Theorem 5]{Garcia-Melian}.

\begin{thm}[\sc Comparison principle]\label{thmCP}
Assume that $(*)$ holds. Let $u,v\in C^2(\Omega)\cap C(\bar \Omega)$ and $u_{|\partial\Omega},v_{|\partial\Omega}\in C^{1,\alpha}(\partial\Omega)$ with $0<\alpha \le 1$, be such that
\begin{equation} \label{comp}
	\left\{ 
	\begin{array}{rcll}
		-\Delta_p u-V|u|^{p-2}u&\le&-\Delta_p v-V|v|^{p-2}v&\text{in $\Omega$,}\\
		u&\le& v&\text{on $\partial\Omega$.}
	\end{array}
	\right.
\end{equation} 
Assume additionally that $-\Delta_p v-V|v|^{p-2}v\ge 0$ in $\Omega$ and $v\ge 0$ on $\partial\Omega$, that is $v$ is a non-negative supersolution in $\Omega$.
Then $u\le v$ in $\Omega$.
\end{thm}
The requirement for $v$ to be a non-negative supersolution in $\Omega$ cannot be dropped unless $V\leq 0$, see \cite[Corollary 8]{Garcia-Melian}. 

Extensions of comparison principles to unbounded domains were developed in \cite{Shafrir,PT,LLM}, all these results include technical assumptions on the decay of the subsolution $u$ and its derivatives at infinity.

In the linear case $p=2$, a comparison principle for Schr\"odinger operators $-\Delta-V$ on bounded domains and with sufficiently regular potentials $V$ can be easily extended to a  Phragm\'{e}n--Lindel\"{o}f type comparison principle, which provides a useful characterisation of admissible asymptotic behaviour of supersolutions decay for instance when $\Omega$ is an exterior domain \cite{LLM}, or when $V$ is very singular on the boundary of a bounded domain \cite{Bandle}. The proof of the Phragm\'{e}n--Lindel\"{o}f type comparison principle relies on the additivity of the Schr\"odinger operator $-\Delta-V$, which is no longer available if $p\neq 2$ and $V\neq 0$. 

The main goal of this work is to show that a partial extension of the Phragm\'{e}n--Lindel\"{o}f type comparison principle is possible in the case $p>2$. Our arguments involve a superposition principle for the $p$--Laplacian that we establish and which provides a partial replacement for the missing additivity. In order to better clarify our contribution, we recall the statement and the proof of a Phragm\'{e}n--Lindel\"{o}f type comparison principle in exterior domains in the linear case $p=2$. 

Note that the Phragm\'{e}n--Lindel\"{o}f principle was originally formulated in \cite{phragmen} for holomorphic  functions defined on arbitrary domains of the complex plane and was applied to a number of cases including unbounded sectors. A corresponding result for harmonic functions in the half space of  $\R^N$ was proved in \cite{ah}, while the case of general linear elliptic partial differential equations was considered in \cite{gilbarg} for two variables and in \cite{hopf} for $N$ variables. The  result for harmonic functions was  extended in \cite{lqvjde} to the case of $p$-harmonic functions. Further generalizations to fully nonlinear equations were obtained in the paper \cite{cavi} to which we refer for more information and references on this subject. We  refer to \cite{PW}, \cite{rudin} for an introduction to this topic.

\subsection{Phragm\'{e}n--Lindel\"{o}f principle in exterior domains in the case $p=2$}

Let $u, v$ be two $C^2$--functions on the exterior of the  unit open ball $B_1$ centered at zero.  We say that $u$ is {\it dominated} by $v$ and write  
$$u\lesssim v\quad\text{in $B_{1}^c$},$$
if  there exists $C>0$ such that $u(x)\le C v(x)$ for all $|x|\ge 1$; and 
$u$ is {\it smaller} than $v$ if
\beq\label{e-small}
\lim_{|x|\to\infty}\frac{u(x)}{v(x)}=0.
\eeq
A simple but essential observation in the spirit of \cite{Agmon} is that if a subsolution is smaller than {\em one} of the positive supersolutions, then  it is dominated by {\em every} positive supersolution.

\begin{lem}\label{lem-small}
	Assume that $p=2$ and $(*)$ holds. Assume that  $u\ge 0$ is a subsolution of \eqref{e1} in $B_{1}^c$ and  that 
	\begin{equation}\label{a-small}
		\text{$\exists$ a positive supersolution $v$ of \eqref{e1} in $B_{1}^c$ such that $u$ is smaller than $v$.}
	\end{equation}
	Then 
	$$u\lesssim w\quad\text{in $B_{1}^c$}$$
	for any positive supersolution $w$ of \eqref{e1} in $B_{1}^c$.
\end{lem}

\begin{proof}

	Fix a positive supersolution $w$ and choose $c>0$ such that $u(x)\le cw(x)$ as $|x|=1$, this is possible since $w(x)>0$ when $|x|=1$.
	For small $\eps>0$ consider the barrier functions  $u_{\epsilon}$ defined by 
	\begin{equation}\label{e-barrier}
	u_\eps =u-\eps v
	\end{equation}
	and observe that $u_\eps$ is a subsolution of \eqref{e1} in $B_{1}^c$, since $p=2$.
	In view of \eqref{e-small}, there exists $R_\eps>1$ such that 
	$$u_\eps\le 0\qquad(|x|\ge R_\eps)$$
	and moreover, $R_\eps\to\infty$ as $\eps\to 0$.
	
	Set $B[1,R_\eps{[}\,=\{x\in\R^N:1<|x|<R_\eps\}$. 
	Since $w>0$, we conclude that
	$$u_\eps\le cw\quad\text{on $\partial B[1,R_\eps[$}.$$
	By the comparison principle (that in the case $p=2$ follows from the maximum principle of Theorem \ref{thmMP} by additivity) we conclude that 
	$$u_\eps\le cw\quad\text{in $B[1,R_\eps[$}.$$
	Then the assertion follows, since $\eps>0$ could be taken arbitrary small and taking into account that $R_\eps\to\infty$ as $\eps\to 0$.
	\end{proof}
	
The following Phragm\'{e}n--Lindel\"of Alternative can be derived as  an immediate corollary of Lemma \ref{lem-small}. Despite its simplicity, it provides a powerful tool for the analysis of sub and supersolutions to linear and nonlinear equations, cf. \cite{LLM,Bandle}. 
	
\begin{prop}[Phragm\'{e}n--Lindel\"of Alternative for $p=2$]\label{PL-2} 
	Assume that $p=2$ and $(*)$ holds. Assume that  $u\ge 0$ is a subsolution of \eqref{e1} in $B_1^c$. 	
	Then  either of the following alternatives occurs: 
	\begin{itemize}
		\item[$i)$]  
		for any positive supersolution $w$ of \eqref{e1} in $B_1^c$, 
		\begin{equation}\label{c-super1-2}
			u\lesssim w\quad\text{in $B_1^c$}
		\end{equation}
		\item[$ii)$] 
		for any positive supersolution $w$ of \eqref{e1} in $B_1^c$, 
		$u$ is not  smaller than $w$, i.e.
		\begin{equation}\label{limitesup-1}
			\limsup_{|x|\to\infty}\frac{u(x)}{w(x)}>0.
		\end{equation}
	\end{itemize} 
\end{prop}

\begin{proof}
$i)$ is simply the reformulation of Lemma \ref{lem-small}, while $ii)$ is the negation of \eqref{a-small}.
\end{proof}

For related results concerning linear equations with possibly changing sign potentials we refer to \cite{vitolo} and the references therein. 

Clearly, when $p\neq 2$ the construction of barriers as in \eqref{e-barrier} is no longer available, unless $V\le 0$. Indeed, if $V\le 0$ then it is possible to use the constant supersolution $v=1$ and obtain the barrier $u_{\epsilon}=cu-\epsilon$ for $\epsilon >0$. It is straightforward that if $V\le 0$ and $u$ is a subsolution of \eqref{e1} then  also $u_{\epsilon}$ is a subsolution. Thus, by applying the same method above one can prove that if $V\le 0$ and $u$ is an infinitesimal subsolution then the same conclusion of Lemma~\ref{lem-small} holds. 

Note that in this case, the second alternative does not hold since \eqref{limitesup-1} is false  for $w=1$ if $u$ is infinitesimal.  This and Remark~\ref{largeVnegative} give criteria to tell if $u$ is a `small' or a `large' subsolution as in the linear case. 

Unfortunately, this argument does not work if $V\geq 0$ since $v=1$ is not a supersolution, unless $V=0$. For this reason, in this paper we focus on potentials $V$ satisfying the sign condition $V\geq 0$. Under this hypothesis, we prove a number of results that we briefly describe in the following lines.

In Section~\ref{section2} we are going to show that some partial additivity for \eqref{e1} can be recovered in a form of a superposition principle  for radial functions, see Theorem~\ref{superposition}.  The proof of  Theorem~\ref{superposition} is based on a convexity argument, see Lemma~\ref{simpleineq}. 
We believe that this result has its own interest since very few superposition results are available in the literature for the $p$-Laplacian, see \cite{bru17, bru20, crajia, linman}. 
 Then in Section 3 we employ this superposition principle in order to prove a restricted version of the Phragm\'{e}n--Lindel\"of Alternative for the $p$-Laplacian with $p\geq 2$, see Theorems~\ref{c-super}, \ref{c-super-2}. 
 The main restriction in our version of this principle consists in  the fact that the  subsolution $u$ is assumed to be either radial and decreasing in Theorem~\ref{c-super} or radial and increasing in Theorem~\ref{c-super-2}. Since our results allow to provide lower and upper bounds for supersolutions not necessarily radial and it is very natural to use radial functions to estimate the growth of functions at infinity, we believe that our assumptions are not very restrictive in applications. In fact, certain supersolutions that appear in our theorems  could be defined on sufficiently smooth outer domains, not necessarily radial, see Remark~\ref{genouter}. In Section 4 we use our results to deduce two--sided bounds on the decay of supersolutions of \eqref{e1} in the case of Hardy type potentials that recover and improve estimates in \cite{LLM}: first we consider the case of pure Hardy potentials in Section~\ref{puresubsec} and then the case of improved Hardy potentials in Section~\ref{improvedsubsec}.  

Finally, we note that by using our results we also recover an estimate proved in \cite[Thm.~5.4]{lampin}, see Remark~\ref{lampinrem} for details.

\section{A superposition principle for $p$-Laplace equations}
\label{section2}

In this section we establish a superposition principle for radial subsolutions and supersolutions of $p$-Laplace equations involving  non-positive potentials. 

Given a radially symmetric domain $\Omega$ in $\mathbb{R}^N$,  we consider   equation \eqref{e1}  with  a non-negative potential $V$  in $\Omega$. 
Recall that for a radially symmetric function $u$ of the type $u(x) =\phi (r)$ where $r=|x|$ and 
$\phi $ is a function of the one real variable $r$, the $p$-Laplacian can be written in spherical coordinates as 
\begin{equation}
\label{mainbis}
\Delta_p(\phi (r) )= |\phi'(r)|^{p-2}L(\phi (r))
\end{equation}
where $L$ is the ordinary differential operator defined by 
\begin{equation}
\label{polar}
L(\phi (r)):= (p-1)\phi''(r) +\frac{N-1}{r}\phi'(r)\, .
\end{equation}

Note that for the sake of simplicity, in what follows we  identify a radial function $u$ as above with $\phi$. Thus, we shall also write $u(r)$ instead of $\phi (r)$, and $u'(r)$ instead of  $\phi'(r)$.   

The notions of solutions, subsolutions and supersolutions of equation \eqref{e1} are standard: usually, one considers functions $u$ in the Sobolev space $W^1_{p,loc}(\Omega)$ and defines those notions by using the appropriate variational weak formulations involving positive test functions. In that framework, keeping in mind the regularity theory for $p$-Laplace equations, one typically assumes that $u$ is a function of class $C^1$ or $C^{1,\gamma}$ for some $\gamma \in ]0,1]$.   However, for technical reasons, we shall  assume that our solutions, sub and supersolutions are of class $C^2$ and the equations will be understood in the classical sense.  We note that this assumption is not much restrictive since here we are concerned with strictly monotone functions, in which case the derivative $\phi '(r)$ does not vanish.  Thus the coefficient $|\phi' (r)|^{p-2}$ in front of the linear operator $L$ in \eqref{mainbis} is locally bounded away from zero and infinity and, in view of classical regularity theory, one expects more regularity for solutions.

The proof of our superposition result is based on the linearity of the operator $L$ above and on a convexity argument involving the nonlinear factor $|\phi '|^{p-2}$. This argument relies on  the following lemma. 

\begin{lem}\label{simpleineq} For all  non-negative real numbers $a,b$ and all positive real numbers $c,d$ the following inequalities hold:
\begin{itemize}
\item[(i)] if $q\geq 1$ then 
\begin{equation}\label{ine1}
\frac{(a+b)^{q}}{(c+d)^{q-1}}\le \frac{a^{q}}{c^{q-1}}  + \frac{b^{q}}{d^{q-1}};
\end{equation}
\item[(ii)] if $0< q\le 1$ then 
\begin{equation}\label{ine2}
\frac{(a+b)^{q}}{(c+d)^{q-1}}\geq \frac{a^{q}}{c^{q-1}}  + \frac{b^{q}}{d^{q-1}}.
\end{equation}
\end{itemize}
Moreover, for $q\ne 1$ equalities occur if and only if $ad=bc$.
\end{lem}

{\bf Proof.}  If $a=0$ or $b=0$ then the statement is trivial. Thus we assume that $a,b>0$. 
We note that,  by setting $x=a/b$ and $y=c/d$,  proving inequalities \eqref{ine1} and \eqref{ine2} is equivalent to proving that 
\begin{itemize}
\item[(i)] if $q\geq 1$ then 
\begin{equation}\label{ine1bis}
\left(\frac{x+1}{y+1}\right)^q     \le    \frac{y}{y+1}\left(   \frac{x}{y}\right)^q  + \frac{1}{y+1};
\end{equation}
\item[(ii)] if $0<q\le 1$ then 
\begin{equation}\label{ine2bis}
\left(\frac{x+1}{y+1}\right)^q     \geq   \frac{y}{y+1}\left(   \frac{x}{y}\right)^q  + \frac{1}{y+1},
\end{equation}
\end{itemize}
respectively. 

We set $\theta =\frac{y}{y+1}$ and we note that   $1-\theta = \frac{1}{y+1}$ and
\begin{equation}
\label{convexcomb}
\frac{x+1}{y+1}= \theta\cdot  \frac{x}{y} +(1-\theta) \cdot 1
\end{equation}

Thus, setting $f(t)=t^q$ for any $t\geq 0$, inequality \eqref{ine1bis} can be written as 
$$
f\left(\frac{x+1}{y+1}\right) \le \theta f\left(\frac{x}{y} \right)+(1-\theta )f(1)
$$
which is true by the convexity of $f$ when $q\geq 1$. Inequality \eqref{ine2bis} follows in the same way by the concavity of $f$ when $0<q\le 1$.  

By the strict convexity and concavity of $f$ for, it follows that the inequalities    \eqref{ine1bis}, \eqref{ine2bis} are strict unless 
$(x+1)/(y+1)=
x/y= 1
$
which means that $x=y$, or equivalently  $ad=bc$.
\qed

We are now ready to prove our superposition result. Note that for this results, it is not required that the subsolutions or the supersolutions are regular up to the boundary of $\Omega$.

\begin{thm}\label{superposition} Let $p\geq 2$, $\Omega$ be a radially symmetric domain in $\mathbb{R}^N$ and $V\geq 0$. Let   $u$, $v$ be radial functions of class $C^2$ in $\Omega$ such that  
\begin{equation}
0< v \le u  \ \ {\rm in}\ \Omega 
\end{equation}
and 
\begin{equation}
 u'\le v' < 0\ \ {\rm in}\ \Omega \ \ {\rm or}\ \ \  u'\ge  v'> 0\ {\rm in}\ \Omega  . 
\end{equation}
If $u$ is a subsolution and $v$ is a supersolution of equation \eqref{e1}  then 
$u-v$ is a subsolution of equation \eqref{e1}. 
\end{thm}

{\bf Proof.} We assume directly that $p>2$ since for $p=2$ the statement is trivial. Since $u$ is a subsolution and $v$ a supersolution, we have that 
\begin{equation} \label{subsup}
\left\{ 
\begin{array}{l}
-L(u ) \le V \frac{u^{p-1}}{|u '|^{p-2}}\\
L(v) \le  -V \frac{v^{p-1}}{|v'|^{p-2}}\, . 
\end{array}
\right.
\end{equation} 
If $x\in\Omega$ is such that $u'(x)=v'(x)$ then it is straightforward that 
 $$
0=  - |u '(x) -v'(x) |^{p-2}L(u(x) -v(x))\le  V(x)   (u(x)-v(x) )^{p-1}
  $$
  because $V(x)\geq 0$. 
  
  Now we consider functions $u$ and $v$ on the open subset of $\Omega$ where $u'\ne v'$. 
By the assumptions on the signs of $u-v$ and $u'-v'$, it is possible to apply  inequality \eqref{ine1}  with $q=p-1$, 
$a=u -v$, $b=v$ and  $c=u' -v'$, $d=v'$  if $u'> v' > 0$  or   $c= v'-u'$, $d=-v'$ if $u'<v' < 0$,  to obtain by \eqref{subsup} that 
\begin{eqnarray}\lefteqn{
-L(u -v)\le V\left( \frac{u^{p-1}}{|u '|^{p-2}}  - \frac{v^{p-1}}{|v'|^{p-2}}\right)  }\nonumber \\
& & 
 = V\left( \frac{  (u-v +v)^{p-1}}{|u ' -v' +v'|^{p-2}}  - \frac{v^{p-1}}{|v'|^{p-2}}\right) \nonumber  \\
 & &  
 \le   V\left( \frac{  (u-v )^{p-1}}{|u ' -v' |^{p-2}} + \frac{v^{p-1}}{|v'|^{p-2}}  - \frac{v^{p-1}}{|v'|^{p-2}}\right)=
  V \frac{  (u-v )^{p-1}}{|u ' -v' |^{p-2}} .
  \end{eqnarray}
 This  implies that 
  $$
  - |u ' -v' |^{p-2}L(u -v)\le  V   (u-v )^{p-1}, 
  $$
hence $u -v $ is a subsolution.  \qed

\begin{rem} Although it is not used in this paper, we note that, in the case $1\le p\le 2$  the following superposition result hold under the same assumption $V\geq 0$.

{\it Assume that
\begin{equation}
0< v \le u  \ \ {\rm in}\ \Omega 
\end{equation}
and 
\begin{equation}
 u'< v' < 0\ \ {\rm in}\ \Omega \ \ {\rm or}\ \ \  u'> v'> 0\ {\rm in}\ \Omega  . 
\end{equation}
If $u$ is a supersolution  and $v$  a subsolution of \eqref{e1}  then
$u-v$ is a supersolution of \eqref{e1}. }

Indeed, by proceeding as in the proof of the previous theorem, we note that 
 since $u$ is a supersolution and $v$ a subsolution, we have that 
\begin{equation} \label{subsup-2}
\left\{ 
\begin{array}{l}
L(u ) \le -V \frac{u^{p-1}}{|u '|^{p-2}}\\
-L(v) \le  V \frac{v^{p-1}}{|v'|^{p-2}}
\end{array}
\right.
\end{equation}
Again, by the assumptions on the signs of $u-v$ and $u'-v'$, it is possible to apply  inequality \eqref{ine2}  with 
$a=u -v$, $b=v$, $c=u' -v'$, $d=v'$  if $u'> v' > 0$  or   $c= v'-u'$, $d=-v'$ if $u'<v' < 0$,  to obtain that 
\begin{eqnarray}\lefteqn{
L(u -v)\le -V\left( \frac{u^{p-1}}{|u '|^{p-2}}  - \frac{v^{p-1}}{|v'|^{p-2}}\right)  }\nonumber \\
& & 
 = -V\left( \frac{  (u-v +v)^{p-1}}{|u ' -v' +v'|^{p-2}}  - \frac{v^{p-1}}{|v'|^{p-2}}\right)\nonumber  \\
 & &  
 \le   -V\left( \frac{  (u-v )^{p-1}}{|u ' -v' |^{p-2}} + \frac{v^{p-1}}{|v'|^{p-2}}  - \frac{v^{p-1}}{|v'|^{p-2}}\right)=
-  V \frac{  (u-v )^{p-1}}{|u ' -v' |^{p-2}} .
  \end{eqnarray}
 This  implies that 
  $$
  - |u ' -v' |^{p-2}L(u -v)\geq  V   (u-v )^{p-1}, 
  $$
hence $u -v $ is a supersolution.
\end{rem}

\section{Phragm\'{e}n-Lindel\"{o}f alternative}

In this section we prove  a version of the Phragm\'{e}n-Lindel\"{o}f alternative: in Theorem~\ref{c-super} we consider the case of decreasing subsolutions, in Theorem~\ref{c-super-2} the case of increasing subsolutions. To do so,  we also need Lemma~\ref{nomax} which allows us to control the monotonicity of  the ratio bewteeen monotone sub and supersolutions.

In what follows, for any $r_0, R_0\in ]0, \infty ]$ with $r_0<R_0$ we denote by $B]r_0, R_0[$  the open annulus defined by 
$$
B]r_0, R_0[=\{x\in \R^N: \ r_0<|x|< R_0 \}
$$
and by  $B[r_0, R_0[$ the semi-closed annulus defined by
$$
B[r_0, R_0[=\{x\in \R^N: \ r_0\le |x|< R_0 \}\, .
$$
As usual, for $r>0$ we set $S_r=\{x\in \R^N: |x|= r\}$.  We note that in this paper we are mainly interested in the case $R_0=\infty$ that is discussed in detail in the next session for Hardy-type potentials. However, the statements in this section include all cases $0<R_0\le \infty$. 

We shall consider equation \eqref{e1} with $\Omega =B]r_0, R_0[ $. In this case, when we say that a function $u$ is a $C^2$-subsolution, supersolution or solution we understand that  $u$ is  of class $C^2(B[r_0,R_0))$. In other words, we require the regularity of $u$ up  to $S_{r_0}$ but not up to $S_{R_0}$ if $R_0$ is finite.

Moreover, we say that $u$ is a strict subsolution of \eqref{e1} if 
$$
-\Delta_p u(x)-V(x)|u(x)|^{p-2}u(x)<0,\ \ \forall x\in \Omega .
$$
Similarly, we say that $v$ is  a strict supersolution 
of \eqref{e1} if
$$
-\Delta_p v(x)-V(x)|v(x)|^{p-2}v(x)>0,\ \ \forall x\in \Omega .
$$

For real-valued radially symmetric $C^1$--functions $u$ defined on  $B[r_0, R_0[$  we say that  $u$ is {\it strictly monotone decreasing} if $u'(r)<0$ for any $r\in[r_0, R_0[$, and $u$ is monotone non--increasing if $u'(r)\le 0$ for any $r\in[r_0, R_0[$. Similarly for monotone strictly monotone increasing/monotone non-decreasing.
It is also convenient for us to introduce further terminology and notation. 

\begin{defn}\label{def-smaller}
	Let $u, v$ be two positive radial $C^1$--functions defined on an annulus $B[r_0, R_0[$.  We say that:
	
	\begin{itemize}
		\item[(i)]
		$u$ is {\it dominated} by $v$ and write  
		$$u\lesssim v\quad\text{in $B[r_0, R_0[$},$$
		if  there exists $C>0$ such that $u(r)\le C v(r)$ for all $r\in[r_0, R_0[$  ;
	
\item	[(ii)]
	$u$ is {\it smaller} than $v$ if
	$$
	u(r)=o(v(r)) \quad {\rm as}\  r\to R_0,
	$$
	i.e. for every $\eps>0$ there exists $R_\eps\in[r_0, R_0[$ such that $u(r)\le \eps v(r)$ for all $r\in[R_\eps, R_0[$;
	
	\item[(iii)]
	$u$ is {\it monotonically smaller} than $v$ if it is smaller than $v$ and there exists $r_1\in [r_0,  R_0 [$ such that $u/v$ is monotone non-increasing  on $[r_1, R_0[$.
	
\end{itemize}
	
\end{defn}

\begin{lem}\label{super} 
	Let  $p\geq 2$, $\Omega = B]r_0, R_0[$ be an annulus  with  $0<r_0<R_0\le \infty $ and  $V\geq 0$. Assume that  condition $(*)$ holds.    Assume that  $u$ is a positive radial strictly monotone decreasing $C^2$-subsolution of \eqref{e1} and that there exists a positive radial strictly monotone decreasing $C^2$-supersolution $v$ of \eqref{e1}  such that $u$ is monotonically smaller than $v$.
	Then 
	$$u\lesssim w\quad\text{in $B[r_0, R_0[$}$$
	for any positive (not necessarily radial or decreasing) $C^2$-supersolution
\end{lem}

\proof  
Let $r_1$ be as in Definition~\ref{def-smaller} (iii). 
 Since $u$ is monotonically smaller than $v$, $R_\eps:=\inf\{\rho\in[r_1, R_0[\;: u(r)\le\eps v(r)\ \forall r\geq \rho\}$ is well-defined and $R_\eps:\R_+\to [r_1, R_0[$ is a non-increasing function of $\eps>0$. 
Moreover, $\lim_{\eps\to 0}R_\eps=R_0$ since $u$ and $v$ are positive in $B[r_0, R_0[$  and $u$ is smaller than $v$. 

Let $\epsilon>0$ be small enough to guarantee that 
$R_{\epsilon}>r_1$.  By the definition of $R_\eps$ and monotonicity we have
\begin{equation*}
0< \eps v \le u \quad\text{in $B[r_1, R_\eps[$}.
\end{equation*}
Further, since $u/v$ is monotone decreasing its derivative is negative and hence
$$v'(r)u(x)\geq v(r)u'(r)\quad\forall r\in[r_1,R_0[.$$
Taking into account that $v'$ is negative, we conclude
$$
\frac{u'(r)}{v'(r)}\geq \frac{u(r)}{v(r)}\ge\eps\quad\forall r\in [r_1, R_{\eps}[.
$$
It follows that 
$$
u'(r) \le  \eps v'(r) < 0 \quad \forall r\in [r_1, R_{\eps}[.
$$
Thus Theorem \ref{superposition} can be used and we conclude that $u-\eps v$ is a nonnegative subsolution of of \eqref{e1} in $B[r_1, R_\eps[$. By construction, we also have that  $u-\eps v=0$ on $S_{R\eps}$.
Let $C>0$ be fixed in such a way that $u\le Cw$ on $B[r_0,r_1]$.
Then $u-\eps v \le Cw$ on $\partial B[r_0, R_\eps[$. Thus, by the Comparison Principle  \cite[Theorem~5]{Garcia-Melian}  (restated in  Theorem~\ref{thmCP})  it follows that 
\begin{equation}\label{inepsilon}
	u(x)-\eps v(x)\le Cw(x)\quad\forall x\in B[r_0, R_\eps].
\end{equation} 
Since $R_{\eps }\to R_0$ as $\eps \to 0$ and $C$ is independent of $\eps$, we can pass to the limit as $\eps \to 0 $ in \eqref{inepsilon} and conclude that $u(x)\le Cw(x)$ for all $x\in B[r_0, R_0[$.
\qed   \\

The proof of the previous theorem is based on the  monotonicity of the function $u/v$. The second step in the proof of the Phragm\'{e}n-Lindel\"{o}f alternative consists in removing that assumption. This is done by means of the following lemma which can be considered as a Maximum Principle for the quotient between a subsolution and a supersolution. 

\begin{lem}[Maximum Principle for the quotient]\label{nomax}
Let $p> 1$ and $\Omega = B]r_0, R_0[$ be an annulus  with  $0<r_0<R_0\le \infty $. Assume that  $u$ is a positive radial strictly monotone $C^2$--subsolution and $v$ is a positive radial strictly monotone $C^2$--supersolution of \eqref{e1}.  Assume in addition 
that either $u$ is a strict subsolution or $v$ is a strict supersolution.  
If $( u/ v)'(\rho)=0$ for some $\rho \in ]r_0,R_0[$ then $( u/ v)''(\rho)>0$. In particular, $u/v$ cannot have a local maximum in $B]r_0, R_0[$. 
\end{lem}

\proof
We consider the function $w$ defined by  
$$w=\frac{u}{v}$$
and assume  that $v$ is a strict supersolution, so that
$$
\begin{array}{l}
	-Lu- V \big(\frac{u}{|u'|}\big)^{p-2}u\le 0\smallskip\\
	-Lv- V \big(\frac{v}{|v'|}\big)^{p-2}v> 0.
\end{array}
$$
Similarly to \cite[p.8]{PW} we compute, 
\begin{multline}
	Lu=L[w v]=(p-1)(wv)''+\tfrac{N-1}{r}(wv)'\\
	=(p-1)vw''+2(p-1)v'w'+(p-1)wv''+\tfrac{N-1}{r}(vw'+v'w)\\
	=(p-1)vw''+\big(2(p-1)v'+\tfrac{N-1}{r}v\big)w'+wLv.
\end{multline}
Then
\begin{multline}
	-Lu- V \frac{u^{p-2}}{|u '|^{p-2}}u=-L[w v]-V \frac{u^{p-2}}{|u '|^{p-2}}(wv)\\
	=-(p-1)vw''-\big(2(p-1)v'+\tfrac{N-1}{r}v\big)w'-w\big(Lv+\tilde{V}v)\le 0,
\end{multline}
where
$$\tilde{V}:=V\Big(\frac{u}{|u'|}\Big)^{p-2}.$$
Dividing by $v>0$, we obtain
\begin{equation}\label{PW1}
	-(p-1)w''-\Big(2(p-1)\frac{v'}{v}+\frac{N-1}{r}\Big)w'-\frac{Lv+\tilde{V}v}{v}w\le 0.
\end{equation}
Assume $w'(\rho)=0$ at $\rho\in ]r_0, R_0[$. Then 
$$\frac{u'(\rho)}{u(\rho)}=\frac{v'(\rho)}{v(\rho)},$$
and
$$\tilde{V}(\rho):=V(\rho)\Big(\frac{u(\rho)}{|u'(\rho)|}\Big)^{p-2}=V(\rho)\Big(\frac{v(\rho)}{|v'(\rho)|}\Big)^{p-2}.$$
Hence
$$Lv(\rho)+\tilde{V}(\rho)v(\rho)=Lv(\rho)+V(\rho)\Big(\frac{v(\rho)}{|v'(\rho)|}\Big)^{p-2}v(\rho)<0.$$
Then from \eqref{PW1} we conclude that
\begin{equation}
	-(p-1)w''(\rho)\le\frac{Lv(\rho)+\tilde V(\rho)v(\rho)}{v(\rho)}w(\rho)<0,
\end{equation}
and hence $w''(\rho)>0$.

The case when $u$ is strict subsolution can be treated exactly  in the same way. 
\qed  

\begin{corol}\label{cornomax}
 Let $p\geq 2$, $\Omega =B]r_0, R_0[$ be an annulus  with  $0<r_0<R_0\le \infty $ and let  either $V>0$ or $V=0$. Assume that  $u$ is a positive radial strictly monotone $C^2$--subsolution and $v$ is a positive radial strictly monotone $C^2$--supersolution of \eqref{e1}.  
 Then at least one of the following cases occurs 
\begin{itemize} 
\item[(i)] either $(u/v)'(r)\le 0$ for all $r\in [r_0, R_0[$
\item[(ii)] or there exists $\rho^*\in [r_0,R_0[$ such that   $(u/v)'(r)\geq 0$ for all $r\in [\rho^*, R_0[$.
\end{itemize}
\end{corol}

\begin{proof}  We begin with the case $V>0$. Assume that there exists $\hat r_0, \rho^*\in ]r_0, R_0[$  with $\hat r_0<\rho^*$ such  that $(u/v)(\hat r_0)< (u/v) (\rho^*)$. 
 We set $u_{\epsilon}(r)=u(r)+\epsilon$ for all $\epsilon >0$ and we observe that since $V>0$ 
$$ 
-\Delta_pu_{\epsilon}=-\Delta_pu\le Vu^{p-1}< Vu_{\epsilon}^{p-1}
$$
hence $u_{\epsilon}$ is a strict subsolution.
Thus, by Lemma~\ref{nomax} the function $u_{\epsilon}/v$ cannot have a local maximum in $]r_0,R_0[$. Let $R\in ]\rho^*, R_0[$ be fixed. Consider the restriction of 
$u_{\epsilon}/v$ to $[\hat r_0, R]$ and call it $w_{\epsilon , R}$. Then, the global maximum of $w_{\epsilon , R}$ is attained either at $\hat r_0$ or at $R$. Assume that for some sequence of positive numbers $\epsilon_n$ with $\epsilon_n\to 0$ we have that $w_{\epsilon_n , R}$ attains the maximum at $\hat r_0$. Then
$w_{\epsilon_n , R}(\hat r_0)\geq w_{\epsilon_n , R}(\rho^*)$ and passing to the limit as $n\to \infty$, we get $(u/v)(\hat r_0)\geq  (u/v)(\rho^*)$ which is a contradiction. 
We conclude that for any $\epsilon$ sufficiently small  $w_{\epsilon , R}$ attains the maximum at $R$. This implies that the 
$w'_{\epsilon_n , R}(R)\geq 0.
$
Passing to the limit in the previous inequality as $\epsilon \to 0$ we get 
$
(u/v)'(R)\geq 0
$ 
for all $R\in  [\rho^*, R_0[$, and case $(ii)$ occurs. 

In order to conclude it suffices to observe that if a couple $(\hat r_0, \rho^*)$ as above does not exists then the function $u/v$ is monotone non-increasing, hence
case $(i)$ occurs. 

In the case $V=0$, the proof can be carried out in the same way provided that the function $u_{\epsilon}(r)=u(r)+\epsilon$ is replaced by  $\tilde u_{\epsilon}(r)=u(r)+\epsilon r^{ \alpha }  $ with $\alpha <\min \{ 0, (p-N)/(p-1)   \}$ if $u$ is monotone decreasing and $\alpha >\max \{ 0, (p-N)/(p-1)   \}$ if $u$ is monotone increasing. Note that $L(r^{\alpha})>0$ for all $r>0$. Moreover from the choice of $\alpha$ and the monotonicity of $u$ it follows that   $|\tilde u_{\epsilon}'|^{p-2}>|u'|^{p-2}$ if $p>2$. 
We conclude that if $p\geq 2$ then 
$
-\Delta_p\tilde u_{\epsilon}<-\Delta_pu\le 0
$
hence $\tilde u_{\epsilon}$ is a strict subsolution as required in this proof.   \end{proof}

We are ready to prove the Phragm\'{e}n--Lindel\"of Alternative for decreasing functions.

\begin{thm}[Phragm\'{e}n--Lindel\"of Alternative for decreasing functions]\label{c-super} 
 Let  $p\geq 2$, $\Omega=B]r_0, R_0[$ be an annulus with  $0<r_0<R_0\le\infty$, and let either  $V>0$ or $V=0$.
Assume that condition $(*)$ holds and that  there exists a positive radial strictly monotone decreasing $C^2$--supersolution $v$ of \eqref{e1}.  

Assume that $u$ is a positive radial strictly monotone decreasing $C^2$--subsolution of \eqref{e1}. 	
Then  either of the following alternatives occurs: 
\begin{itemize}
	\item[(i)]  
	for any positive (not necessarily radial or decreasing)  $C^2$-supersolution $w$ of \eqref{e1}
	\begin{equation}\label{c-super1}u\lesssim w\quad\text{in $B[r_0, R_0[$}\end{equation}
	\item[(ii)] 
	for any positive radial strictly monotone decreasing $C^2$--supersolution $v$ of \eqref{e1}, 
	$u$ is not  smaller than $v$, i.e.
		\begin{equation}\label{limitesup}\limsup_{r\to R_0}\frac{u(r)}{v(r)}>0.
		\end{equation}
\end{itemize} 
\end{thm}

{\bf Proof.} Assume that alternative (i) does not occur. Then, by Lemma~\ref{super} it follows that 
for any positive radial strictly monotone decreasing $C^2$--supersolution $v$ of \eqref{e1} in $B[r_0, R_0[$, 
	$u$ is not monotonically smaller than $v$, i.e.
	\begin{itemize}
		\item[(1)] either \eqref{limitesup} holds
		
		\item[(2)] or $u$ is smaller than $v$ and $u/v$ is not monotone non-increasing in any neighborhood of $R_0$.
	\end{itemize} 
In oder to conclude the proof it suffices to prove that $(2)$ doesn't occur. Note that by Corollary~\ref{cornomax} if $u/v$ is  not monotone non-increasing in any neighborhood of $R_0$ then  $u/v$ must be  monotone non-decreasing in a neighborhood of $R_0$ but this is not compatible with the fact that $u$ is smaller than $v$. Thus case $(2)$ has to be excluded and the proof is complete. 
 \hfill $\Box$\\

%

\begin{rem}  We can prove Theorem~\ref{c-super} under the more general assumption  $V\geq 0$ provided that $u$ is required, in addition,  to be a strict subsolution. In fact, if $u$ is a strict subsolution then Lemma~\ref{nomax} is directly applicable without using Corollary~\ref{cornomax}. Namely, by  Lemma~\ref{nomax} and the fact that $u$ is a strict subsolution we deduce  that $u/v$ cannot have local maxima, hence in the argument used in the proof of Theorem~\ref{c-super}  the case (2) (where we say that  $u$ is smaller than $v$ and $u/v$ is not monotone non-increasing in any neighborhood of $R_0$) has to be excluded. 
\end{rem}

\begin{rem}\label{genouter}  As it is stated in Theorem~\ref{c-super}, the functions $w$ admissible in statement (i), formula \eqref{c-super1} are not necessarily radial  symmetric. Actually, by taking for simplicity $R_0=\infty$ and  inspecting the proof of Theorem~\ref{c-super} one can easily see that the alternative formulated in statement (i) of Theorem~\ref{c-super} concerns all  supersolutions $w$ defined in any outer  domain $\widetilde \Omega$ contained in  $\Omega$, with sufficiently regular boundary, and such that condition $(*)$ holds in $\widetilde\Omega$. To prove this, it is enough to apply the Comparison Principle to $u$ and $w$ in domains of the form $\tilde \Omega \cap B(0, R]$ with $R$ large enough. Moreover, the regularity assumptions required on $w$ (interior and at the boundary $\partial\widetilde\Omega$) are the minimal assumptions that guarantee the application of the Comparison Principle, and in this sense it is not required that $w$ is a classical supersolution. 
The same observation applies to Theorem~\ref{c-super-2} statement (ii). 
\end{rem}

 We plan now to prove the Phragm\'{e}n--Lindel\"of Alternative for increasing  functions. To do so, we need the following

\begin{lem}\label{principle-2}   
 Let $p\geq 2$, $\Omega =B]r_0, R_0[$ be an annulus with  $0<r_0<R_0\le\infty$ and $V\geq 0$.  Assume that condition $(*)$ holds and that  $u$ is a radial strictly monotone increasing $C^2$--subsolution of \eqref{e1}  with $u(r_0)\ge 0$. Assume that either $u(r_0)=0$ or there exists a positive radial strictly monotone increasing  $C^2$--supersolution $v$ of \eqref{e1}  such that  
\begin{equation}\label{principle0}
u(r_0)=v(r_0),\ \ {\rm and}\ \ \Big(\frac{u}{v}\Big)^\prime>0
\ \ {\rm for\ all}\ r\in ]r_0, R_0[.
\end{equation}

Then for any positive (not necessarily radial or increasing) $C^2$-supersolution $w$ of \eqref{e1}  we have 
		\begin{equation}\label{principle1}\limsup_{|x|\to R_0}\frac{u(x)}{w(x)}>0.
		\end{equation}
\end{lem}

{\bf Proof.}  Assume that $u(r_0)>0$. Let $v$ be as in the statement.
Note that 	$u'(r)>v'(r)$ and $u(r)>v(r)$ for all $r\in]r_0,R_0[$, in view of the monotonicity of $\frac{u}{v}$.

Thus we can apply Theorem~\ref{superposition} and deduce that the function
$u-v$ is a subsolution of \eqref{e1} in $B[r_0, R_0[$. 

Let $w$ be a supersolution as in the statement.  We claim that 
\begin{equation}\label{principle2}\limsup_{|x|\to R_0}\frac{u(x)-v(x)}{w(x)}>0.
		\end{equation}
Note that \eqref{principle2} implies 	 \eqref{principle1}. In order to prove our claim, we argue by contradiction and assume that \eqref{principle2} doesn't hold, that is 
\begin{equation}\label{principle3}\limsup_{|x|\to R_0}\frac{u(x)-v(x)}{w(x)}=0.
		\end{equation}
		It follows that 
\begin{equation}\label{principle4}\liminf_{|x|\to R_0}\frac{w(x)}{u(x)-v(x)}=\infty.
		\end{equation}
	By \eqref{principle4} it follows that for any $n\in \mathbb{N}$ there exists $R_n\in ]r_0, R_0[$ such that 
	\begin{equation}
	\label{wnuv}
	w(x)\geq n  (u(x)-v(x))
	\end{equation}
	for all  $x\in B[R_n, R_0[$.  Now, since $w$ is a supersolution, $n  (u-v)$ is a subsolution and $w\geq n  (u-v) $ at the extremes of  the interval $[r_0, R_n]$, by the  Comparison Principle \cite[Theorem~5]{Garcia-Melian}  (restated in  Theorem~\ref{thmCP})  we deduce that \eqref{wnuv} holds also for all $x\in B[r_0, R_n[$, hence \eqref{wnuv} 
	holds for all $x\in B[r_0, R_0[$.  Letting $n\to \infty $ in \eqref{wnuv}  we get a contradiction and the claim is proved. 

To complete the proof, it remains to consider the simpler case $u(r_0)=0$ for which it suffices to repeat directly the same argument with $u$ replacing $u-v$. 
\hfill $\Box$

\begin{rem}\label{largeVnegative} Assume that $V\le 0$ and that $u$ is a positive subsolution of \eqref{e1}  non necessarily radial. Assume that $u(x)\geq c_0$ for all $|x|\geq r_0$ and  $u(x)= c_0$ for all $|x|= r_0$ where $c_0$ is a non-negative constant. Then \eqref{principle1} holds for all positive supersolutions $w$, not necessarily radial.  This can be proved as in the proof of Lemma~\ref{principle-2}: take $v=c_0$, observe that since $V\le 0$ then $u-v$ is a subsolution, then follow the contradiction argument from \eqref{principle2} to \eqref{wnuv}.
\end{rem}

 We now prove  the Phragm\'{e}n--Lindel\"of Alternative for increasing  functions.

\begin{thm}[Phragm\'{e}n--Lindel\"of Alternative for increasing functions]\label{c-super-2} 
	 Let $p\geq 2$,  $\Omega=B]r_0, R_0[$ be an annulus with  $0<r_0<R_0\le\infty$ and $V\geq 0$.
	Assume\footnote{This assumption implies the validity of condition $(*)$} that there exists a positive radial strictly monotone increasing  $C^2$ strict supersolution $v$ of \eqref{e1}.
	
	Assume that $u$ is a radial monotone increasing $C^2$--subsolution of \eqref{e1}  such that $u(r_0)\ge 0$. 	
	Then  either of the following alternatives occurs: 
	\begin{itemize}	
		\item[(i)] 
		for any positive radial strictly monotone increasing $C^2$ strict supersolution $v$ of \eqref{e1}  we have
	\begin{equation}\label{limitesupbis}
		u(r)\lesssim v(r)\quad\text{for all $r\in[r_0,R_0)$}
	\end{equation}
		\item[(ii)]  
		for any positive (not necessarily radial or increasing) $C^2$-supersolution
		$w$ of \eqref{e1},  $u$ is not smaller than $w$, i.e.
		\begin{equation}\label{principle1bis}
			\limsup_{|x|\to R_0}\frac{u(x)}{w(x)}>0.
		\end{equation}
	\end{itemize} 
\end{thm}

{\bf Proof.}
If $u(r_0)=0$ then $(ii)$ holds by Lemma \ref{principle1}. Assume $u(r_0)>0$ and statement $(ii)$ does not hold.    
Let $v$ be a  positive radial strictly monotone increasing  $C^2$ strict supersolution $v$ of \eqref{e1} in $B[r_0, R_0[$. Assume that  
$u(r_0)=v(r_0)$. Then  by Lemma \ref{principle1}  the inequality
\begin{equation}\label{principle0-not}
\Big(\frac{u}{v}\Big)^\prime>0
	\ \ {\rm for\ all}\ r\in ]r_0, R_0[
\end{equation} 
fails, that is there exists at least one point $\hat \rho \in ]r_0, R_0[$ such that $(u/v)^\prime(\hat \rho )\le 0$.

Note that by Lemma \ref{nomax}, function $u/v$ has no critical points unless they are strict minima. This, combined with the fact that   $(u/v)^\prime(\hat \rho )\le 0$, implies the following two possibilities: 
\begin{itemize} 
\item[(a)] either  $u/v$ is a decreasing function 
\item[(b)] or  there exists $\rho^*\in ]r_0,R_0[$ such that $(u/v)'(r)<0   $ for all $r\in ]r_0, r^*[$ and $(u/v)'(r)>0   $ for all $r\in ] r^*, R_0[$, hence $(u/v)'(\rho^*)=0$.
\end{itemize}

If $(a)$ occurs then $u(r)/v(r)\le u(r_0)/v(r_0)=1$ and the inequality $ u(r)\le v(r)$ holds for all $r\in[r_0,R_0)$. Assume now that $(b)$ occurs. We claim that $u(r)< v(r)$ for all $r\in]r_0,R_0[$. Indeed, if by contradiction there exists $\hat r_0\in ]r_0,R_0[$ such that $u(\hat r_0)=v(\hat r_0)$ then $\hat r_0>r^*$   and we can apply  Lemma \ref{principle-2} with $r_0$ replaced by $\hat r_0$ in order to conclude that statement $(ii)$ holds, a contradiction since statement $(ii)$ was negated from the very beginning of the proof. 

If the assumption $v(r_0)=u(r_0)$ is not satisfied then it suffices to replace $v$ by the rescaled function $vu(r_0)/v(r_0)$ in the previous argument in order to conclude that 
 \eqref{limitesupbis} holds. \hfill $\Box$

\section{The case of the Hardy potential}
\label{hardysec}

In this section we apply Theorems~\ref{c-super} and \ref{c-super-2}  to equation \eqref{e1} with Hardy potentials and improved Hardy potentials. We begin with the first case. 

\subsection{Pure Hardy case.}
\label{puresubsec}
We consider the case where the potential $V$ is given by 
$$
V(x)=\frac{\lambda}{|x|^p},
$$
where $\lambda \geq 0$ and $\Omega=B[r_0,\infty)$ with $r_0>0$. Thus equation \eqref{e1} becomes
\begin{equation}\label{mainhardy}
-\Delta_p u-\frac{\lambda}{|x|^p} |u|^{p-2}u=0\quad \text{in $B[r_0,\infty)$}.
\end{equation}
It is convenient to set 
  $$
  \lambda_{\alpha}= \alpha |\alpha |^{p-2}\big(p-N  -(p-1)\alpha   \big)
  $$
  for all $\alpha \in {\mathbb{R}}$. It is a direct computation using formula \eqref{polar} to prove the following:

\begin{itemize}
\item  if $\lambda_{\alpha}\le \lambda$ then $u=r^\alpha$ is a subsolution of the equation in \eqref{mainhardy} in $\mathbb{R}^N\setminus \{0\}$.
\item  if $\lambda_{\alpha}\geq \lambda$ then $u=r^\alpha$ is a supersolution of the equation in  \eqref{mainhardy}   in $\mathbb{R}^N\setminus \{0\}$.
\end{itemize}
Now, let us consider  $\lambda_{\alpha}$ as a function of $\alpha$. It is straightforward to see that  
 the function $\lambda_{\alpha}$ is increasing for $\alpha \in ]-\infty , \frac{p-N}{p}[$ and  decreasing for  $\alpha \in ] \frac{p-N}{p}, +\infty [$ and its maximum is given by 
\begin{equation}\label{ch}
C_H:=\lambda_{\frac{p-N}{p}}= \left| \frac{p-N}{p}   \right|^p\, .
\end{equation}

Assume that $ p\ne N$, so that  $C_H>0$. If $\lambda \in [0, C_H]$ then there exist two solutions $\alpha_{\lambda}$ and $\bar\alpha_{\lambda}$ to the equation
$$
\lambda_{\alpha}=\lambda,
$$
with  the following properties: if $p<N$ then  $(p-N)/(p-1)\le \alpha_{\lambda}\le \bar\alpha_{\lambda}\le 0$,   if $p>N$ then  $0\le \alpha_{\lambda}\le \bar\alpha_{\lambda}\le (p-N)/(p-1)$.  In both cases  $\alpha_{\lambda}=\bar\alpha_{\lambda}$  if and only if $\lambda = C_H$. Moreover, the extremes $(p-N)/(p-1)$ and $0$ are reached by $\alpha_{\lambda}$ and $\bar\alpha_{\lambda}$ only for $\lambda=0$. In the case $p=N$ we have that $C_H$=0 and $\alpha_{\lambda}=\bar\alpha_{\lambda}=0$.
As an immediate  consequence of the previous observations we conclude that for any $\lambda \in [0, C_H]$:
\begin{itemize}
\item[(i)]  If $\alpha \le \alpha_{\lambda} $ or $\alpha \geq \bar\alpha_{\lambda}$ then $u=|x|^\alpha$ is a subsolution of the equation in \eqref{mainhardy} in $\mathbb{R}^N\setminus \{0\}$.
\item[(ii)]  If $   \alpha_{\lambda}   \le \alpha\le  \bar\alpha_{\lambda}   $  then  $u=|x|^\alpha$ is a supersolution of the equation in \eqref{mainhardy}   in $\mathbb{R}^N\setminus \{0\}$.
\end{itemize}
Moreover, the case of strict sub or super solutions occur precisely when $\alpha\ne \alpha_{\lambda}, \bar\alpha_{\lambda}$.

We are now ready to show that Theorem~\ref{c-super} is applicable to equation \eqref{mainhardy} with $\lambda \in [0, C_H]$  in any exterior domain $B[r_0, \infty [$.

\begin{thm} \label{tH} Let $2\le  p< N$ and $\lambda \in [0, C_H]$. 
	Then for any positive  (not necessarily radial or decreasing)  $C^2$-supersolution $w$ of \eqref{mainhardy}  we have
	\begin{equation}\label{tH1}|x|^{\alpha_\lambda}\lesssim w\quad\text{in $B[r_0, \infty [$}.\end{equation}
	Further, for any positive radial strictly monotone decreasing $C^2$-supersolution $v$ of \eqref{mainhardy}  
	if $\lambda<C_H$ we have
	\begin{equation}\label{tH2}\limsup_{|x|\to \infty}\frac{|x|^{\bar\alpha_\lambda}}{v(x)}>0,
	\end{equation}
	 or if $\lambda= C_H$ then for any $\tau>0$ we have
\begin{equation}\label{tH2-log}
	\limsup_{|x|\to \infty}\frac{|x|^{\frac{p-N}{p}}\log^{2/p}|x|\log^\tau\log |x|}{v(x)}>0.
\end{equation}
\end{thm} 

\proof
First assume $\lambda<C_H$.
Fix $\beta\in ]\alpha_{\lambda}, \bar\alpha_{\lambda} [$ and consider the supersolution defined by $v(x)=|x|^{\beta}$. Note that function $v$ guarantees the validity of condition $(*)$. Then the solution $u=|x|^{\alpha_\lambda}$ is monotonically smaller than the supersolution $v$ and by Lemma \ref{super} we conclude that \eqref{tH1} holds for any positive supersolution $w$.

If $\lambda=C_H$ we have $\alpha_{\lambda}=\bar\alpha_{\lambda}$ and the previous argument is not applicable. 
Instead, consider  
$$v(x)=|x|^{\alpha_\lambda}\log^\beta| x|,$$    
in a subdomain $B[\rho_0, \infty )[$ with $\rho_0$ sufficiently large in order to guarantee that 
$v'(r)<0$ for all $r\geq \rho_0$. Then by a direct computation (see Table~\ref{table}), $v$ is a supersolution of \eqref{mainhardy} in $B[\rho_0, \infty )$ if
$\beta\in [0,2/p[$ and $\rho_0$ is sufficiently large. Thus we can apply the previous argument with $\beta\in ]0, 2/p]$ in $B[\rho_0, \infty )$ to conclude that  \eqref{tH1} holds in $B[\rho_0,\infty [$ hence also in $B[r_0,\infty [$ for any positive supersolution $w$.

To establish the bound \eqref{tH2}, we take the solution  $\bar u(x)=|x|^{\bar\alpha_\lambda}$ and note that
if  $0<\lambda<C_H$ then  the function $v(x)=|x|^{\beta}$  used above satisfies the condition   $v=o(\bar u)$ as $r\to\infty$ and hence alternative (i) in Theorem \ref{c-super} can not hold. Thus \eqref{tH2} holds for a generic $v$ as in the statement. If $\lambda=0$ then inequality \eqref{tH2} is trivial because the numerator in \eqref{tH2} is identically equal to $1$ and the denominator is assumed to be monotone decreasing hence bounded.

If $\lambda=C_H$ we take the subsolution $\bar u(x)=|x|^{\bar\alpha_\lambda}\log^{2/p}|x| \log^\tau \log |x| $ with $\tau>0$  in a subdomain $B[\rho_0, \infty )[$ with $\rho _0$ sufficiently large (see Table~\ref{table}) and the solution 
$v(x)=|x|^{\frac{p-N}{p}}$ and note again that $v=o(\bar u)$ as $r\to\infty$ and hence alternative (i) in Theorem \ref{c-super} cannot hold.
%
%
%
%
\qed  

We  now apply Theorem~\ref{c-super-2}  to equation \eqref{mainhardy} with $\lambda \in [0, C_H]$  in any exterior domain $B[r_0, \infty [$.

\begin{thm} \label{tHb} Let $p\in [2,\infty [$ be such that $p>N$ and let $\lambda \in [0, C_H]$.
	Then for any positive  radial strictly monotone increasing $C^2$ strict  supersolution $v$ of \eqref{mainhardy}  we have
	\begin{equation}\label{tH1b}|x|^{\alpha_\lambda}\lesssim v\quad\text{in $B[r_0, \infty [$}.\end{equation}
	Further, for any positive (not necessarily radial or increasing) $C^2$-supersolution $w$ of \eqref{mainhardy}, if $\lambda <C_H$ we have
	\begin{equation}\label{tH2b}\limsup_{|x|\to \infty }\frac{|x|^{\bar\alpha_\lambda}}{w(x)}>0,
	\end{equation}
	 or if $\lambda= C_H$  then for any $\tau>0$ we have
\begin{equation}\label{tH2-logb}
	\limsup_{|x|\to \infty}\frac{|x|^{\frac{p-N}{p}}\log^{2/p}|x|  \log^\tau\log |x|}{w(x)}>0.
\end{equation}
 \end{thm} 

\proof The proof is similar to the proof of Theorem~\ref{tH}.    
First assume $0<\lambda<C_H$. Consider  the  solution $u(x) =|x|^{\alpha_\lambda}$.
Fix $\beta\in ]\alpha_{\lambda}, \bar\alpha_{\lambda} [$ and consider the supersolution defined by $w^*(x)= |x|^{\beta} $. Note that function $w^*$ guarantees the validity of condition $(*)$.
Since inequality \eqref{principle1bis} is not satisfied by the couple of functions $u,w^*$ then alternative $(i)$ in Theorem~\ref{c-super-2} holds hence estimate \eqref{tH1b} is satisfied for any $v$ as in statement. For $\lambda =0$ estimate \eqref{tH1b} is trivial since the left-hand side is constant and $v$  is increasing.  
If $\lambda=C_H$ we have $\alpha_{\lambda}=\bar\alpha_{\lambda}$ and the previous argument is not applicable. 
Again, as in the proof of Theorem~\ref{tH}, we consider  the supersolution
$$w^*(x)=|x|^{\alpha_\lambda}\log^\beta| x|,$$    
with $\beta\in]0,2/p[$  and repeat the argument to conclude that  \eqref{tH1b} holds for any $v$ as in the statement. 

To establish \eqref{tH2b}, we take the solution  $\bar u (x)=|x|^{\bar\alpha_\lambda}$, and  the strict supersolution defined by $v^*(x)= |x|^{\beta}$ where   $\beta\in ]\alpha_{\lambda}, \bar\alpha_{\lambda} [$ is fixed.  Since \eqref{limitesupbis} is not satisfied by the couple of functions $\bar u, v^*$ then alternative $(i)$ in Theorem~\ref{c-super-2} doesn't hold hence \eqref{tH2b} holds for any $w$ as in the statement. 
Finally, in order to prove \eqref{tH2-logb} we repeat the same argument in a subdomain  $B[\rho_0,\infty)$ with $\rho_0$ sufficiently big,  with  the subsolution $\bar u(x)=|x|^{\frac{p-N}{p}}\log^{2/p}|x|(\log\log |x|)^\tau$ with $\tau >0$ and the strict supersolution $v^*(x)=|x|^{\frac{p-N}{p}} \log^{\beta}|x|  $ with $\beta \in ]0,2/p[$, see Table~\ref{table}. 
\qed

\begin{rem} If $p=N\geq 2$ the statement of Theorem~\ref{tHb} holds with the following modifications. First of all we note that if $p=N$ then $\lambda=C_H=0$ hence the case $0\le \lambda <C_H$ is excluded.
Moreover, it turns out that  $\alpha_{\lambda}=\bar\alpha_{\lambda}=0$ hence inequality \eqref{tH1b} is trivial.  On the other hand, inequality 
\eqref{tH2-logb} has to be replaced by the inequality 
\begin{equation}\label{tH2-logbbis}
	\limsup_{|x|\to \infty}\frac{\log |x| }{w(x) }>0.
\end{equation}
This can be proved in a similar way.  Namely, we consider the subsolution  $\bar u(x)=\log |x|$ and the strict supersolution 
$v^*(x)=\log  ^{\beta} |x|$ with $0<\beta < 1$, see Table~\ref{table} (note that in this case $\beta_0=0$ and $\bar\beta_0=1$, see the next section). Since \eqref{limitesupbis} is not satisfied by the couple of functions $\bar u, v^*$ then alternative $(i)$ in Theorem~\ref{c-super-2} doesn't hold hence estimate \eqref{tH2-logbbis}
holds for any $w$ as in the statement.  
\end{rem}

\begin{rem}\label{lampinrem}
Here we highlight the connection between our results and some results proved in \cite{lampin}. 
For this purpose, we rescale the exponents in our estimates by setting   
$$
\alpha = \frac{p-N}{p-1}\beta,
$$
for any $\beta \in \mathbb{R}$.  We note that 
$$
\lambda_{  \frac{p-N}{p-1}\beta}= \left|\frac{p-N}{p-1} \right|^p  \mu_{\beta}
$$
where 
$$
\mu_{\beta}= (p-1)\beta |\beta |^{p-2}(1-\beta)\ . 
$$
Note that the quantity $\mu_{\beta }$ is actually denoted by $\lambda_{\beta}$ in  \cite{lampin}. Let  $\lambda\in [0,C_H]$ be fixed. In order to compare our estimates with those in \cite{lampin}, we  need also to  consider the  exponents 
 $\beta_1\in [(p-1)/p,1]$  and $\beta_2\in [0,(p-1)/p] $ defined by 
$$
\mu_{\beta_1}=\mu_{\beta_2}=\left|\frac{p-1}{p-N}\right|^p\lambda  .
$$
Note that $\beta_2\le\beta_1$ as in \cite{lampin}, hence  $\alpha_1= ((p-N)/p)\beta_1$ and $\alpha_2= ((p-N)/p)\beta_1$ satisfy the inequalities: $\alpha_1\le \alpha_2$ if $p<N$, and $\alpha_1\geq \alpha_2$ if $p>N$.

If  $2\le  p< N$ then by Theorem~\ref{tH} it follows that  for any positive  (not necessarily radial or decreasing)  supersolution $w$ of \eqref{mainhardy} in $B[r_0, \infty [$ we have
	\begin{equation}\label{tH1bis}|x|^{\frac{\beta_1(p-N)}{p-1}}\lesssim w\quad\text{in $B[r_0, \infty [$}.\end{equation}

 If $ p> \max\{N,2\}$ then  by Theorem~\ref{tHb} it follows that  for any positive  radial strictly monotone increasing $C^2$ strict  supersolution $v$ of \eqref{mainhardy} in $B[r_0, \infty [$ we have
	\begin{equation}\label{tH1btris}|x|^{\frac{\beta_2(p-N)}{p-1}} \lesssim v\quad\text{in $B[r_0, \infty [$}.\end{equation}
	
Consider now the equation
\begin{equation}	\label{hardydistance}
	-\Delta_pu-\frac{\lambda}{\delta^p}u^{p-1}=0
\end{equation}
in an exterior domain $\Omega$ (not necessarily symmetric), where $\delta (x)$ denotes the Euclidean distance of $x\in {\mathbb{R}}^N$ to $\partial \Omega$.  Here by exterior domain we mean  the complement in $\mathbb{R}^N$ of a sufficiently regular compact set $K$.  Without loss of generality assume that $0$ belongs to the interior of $K$.  Under these assumptions,  we have  that 
$$
\delta (x) < |x|
$$
for all $x\in \Omega$. It follows that  if $u$ is a positive solution to \eqref{hardydistance} then $u$ is a strict supersolution to \eqref{mainhardy} in  $B[r_0, \infty [$ for any $r_0$ such that   $B[r_0, \infty [\subset \Omega$.

In particular,  if  $2\le  p< N$ then  $u$ satisfies  \eqref{tH1bis}  ($w$ replaced by $u$) hence we retrieve the estimate  for $u$ proved   in \cite[Thm.~5.4~ (ii)]{lampin} by a different  method. 

 If $ p> \max\{N,2\}$ and  $u$ is in addition assumed to be  radial and strictly monotone increasing then  $u$ satisfies  \eqref{tH1btris}  ($v$ replaced by $u$)   hence we retrieve the estimate  for $u$ proved   in \cite[Thm.~5.4~ (iii)]{lampin} by another method. 
	
	We note that in    \cite[Cor.~5.5]{lampin} it is proved that the so-called solution of minimal growth of equation \eqref{hardydistance} in $\Omega$ is asymptotic  for $|x|\to \infty$ to 
	$|x|^{\frac{\beta_1(p-N)}{p-1}}$  if $p<N$ and to $|x|^{\frac{\beta_2(p-N)}{p-1}}$  if $p>N$ and this can be interpreted by using  our alternatives in Theorem~\ref{tH} and Theorem~\ref{tHb} which `classify' the functions  $|x|^{\frac{\beta_1(p-N)}{p-1}}$  if $p<N$ and to $|x|^{\frac{\beta_2(p-N)}{p-1}}$  if $p>N$ as `small solutions' of equation \eqref{mainhardy}.
\end{rem}

\subsection{Improved Hardy potential case.}
\label{improvedsubsec}

The so--called ``improved'' Hardy potential  is
$$V(x)=\frac{C_H}{|x|^p}+\frac{\eps}{|x|^p\log^{m_*}|x|},$$
where $\eps>0$, $m_*=2$ if $p\ne N$, and  $m_*=N$ if $p=N$.  Recall that $C_H=|(p-N)/p|^p$.
The corresponding equation is 
\begin{equation}\label{mainhardy-impr}
	-\Delta_p u-\frac{C_H}{|x|^p} |u|^{p-2}u-\frac{\eps}{|x|^p\log^{m_*}|x|}|u|^{p-2}u=0\quad \text{in $B[r_0,\infty)$}.
\end{equation}

The natural family of sub and supersolutions for this equation is
$$u_{\beta,\tau}(r):=r^{\frac{p-N}{p}}(\log r)^\beta(\log\log r)^\tau,$$
for suitable values of  $\beta, \tau\in\R$ that we are going to describe.
Following \cite{LLM}, we set $C_*=\frac{p-1}{2p}\left|\frac{p-N}{p}\right|^{p-2}$ if $p\ne N$ and $C_*=\left(\frac{N-1}{N}\right)^N$ if $p=N$.
If $\eps\in[0,C_\ast]$ we denote
by $\beta_\eps, \bar\beta_\eps$ with $\beta_\eps\le \bar\beta_\eps$,  the real roots of the equation
\begin{eqnarray}\label{e:80><N}
	\tfrac{1}{2}\left|\tfrac{p-N}{p}\right|^{p-2}(p-1)(2-\beta p)\beta=\eps,\ \ \ {\rm if}\ p\ne N\\
	(N-1)(1-\beta) \beta | \beta |^{N-2}=\epsilon,\ \ \ {\rm if}\ p= N.
\end{eqnarray}
For $p\ne N$ we have  $0<\beta_\eps<\frac{1}{p}<\bar\beta_\eps<\frac{2}{p}$  if $0<\eps<C_*$, 
 $\beta_\eps=\bar\beta_\eps=\frac{1}{p}$ if $\eps=C_*$, and  $\beta_\eps=0, \bar\beta_\eps=\frac{2}{p}$ if $\epsilon =0$.  
 For $p=N$ we have  $0<\beta_\eps<\frac{N-1}{N}<\bar\beta_\eps< 1$  if $0<\eps<C_*$, 
 $\beta_\eps=\bar\beta_\eps=\frac{N-1}{N}$ if $\eps=C_*$, and  $\beta_\eps=0, \bar\beta_\eps=1$ if $\epsilon =0$. 
 
Table~\ref{table}, that is taken from \cite[Table 1]{LLM} (here we have included all cases $\eps\in [0,C_*]$ in the first row)  summarizes the properties of the family $u_{\beta,\tau}$ with respect to the equation
\eqref{mainhardy-impr}. The radius $r_0>1$ is assumed to be chosen sufficiently large since in some cases this is necessary to guarantee that we have in fact a subsolution or a supersolution.
The proofs were obtained by tedious direct computations, see \cite[Lemma A.1 and (A.6)]{LLM} (note that the case $\epsilon\in]0, C_*]$ in the first row can be analyzed in the same way).

\begin{table}[H]
	\begin{center}
		\begin{footnotesize}
			\begin{tabular}{|l|c|c|c|}
				\hline
				& {\sl Subsolution} & {\sl Supersolution} & {\sl Subsolution}\\
				\hline
				$p=N$,  $\eps\in [0,C_*]$ &
				$\beta\le \beta_\eps\,$, $\tau=0$ &
				$\beta\in[\beta_\eps,\bar\beta_\eps]$, $\tau=0$ &
				$\beta\ge\bar\beta_\eps\,$, $\tau=0$\\
				\hline
				$p\neq N$, $\eps=0$ &
				$\beta\le 0\,$, $\tau=0$ &
				$\beta\in[0,2/p[$, $\tau=0$ &
				$\beta>2/p\,$, $\tau=0$\\
				&
				&
				&
				$\beta=2/p\,$, $\tau>0$\\
				\hline
				$p\neq N$, $\eps\in(0,C_\ast)$ &
				$\beta<\beta_\eps\,$, $\tau=0$ &
				$\beta\in(\beta_\eps,\bar\beta_\eps)$, $\tau=0$ &
				$\beta>\bar\beta_\eps\,$, $\tau=0$\\
				&
				$\beta=\beta_\eps\,$, $\tau<0$ &
				$\beta=\beta_\eps$, $\tau>0$ or $\beta=\bar\beta_\eps$, $\tau<0$ &
				$\beta=\bar\beta_\eps\,$, $\tau>0$\\
				\hline
				$p\neq N$, $\eps=C_\ast$ &
				$\beta<1/p\,$, $\tau=0$ &
				&
				$\beta>1/p\,$, $\tau=0$\\
				&
				$\beta=1/p\,$, $\tau<0$ &
				$\beta=1/p\,$, $\tau\in(0,2/p)$&
				$\beta=1/p\,$, $\tau>2/p$\\
				\hline
			\end{tabular}
			\caption{{\em 
					Properties of
					$u_{\beta,\tau}=r^{\frac{p-N}{p}}(\log r)^\beta(\log\log r)^\tau$,
					for a large $r_0>1$.}}\label{table}
		\end{footnotesize}
	\end{center}
\end{table}

We are now ready to show that  Lemma~\ref{super}  and Theorem~\ref{c-super} are applicable to equation \eqref{mainhardy-impr} with $\eps \in ]0, C_*]$  in any exterior domain $B[r_0, \infty [$.

\begin{thm} \label{tHimpr}
	 Let $2\le p<N$, $\eps \in ]0, C_*]$ and $r_0>e$. Then for any positive  (not necessarily radial)  $C^2$-supersolution $w$ of \eqref{mainhardy-impr}  we have that 
	\begin{equation}\label{tHimpr1}|x|^\frac{p-N}{p}\log^{\beta_\eps}|x|\log\log^\tau|x|\lesssim w\quad\text{in $B[r_0, \infty [$},\end{equation}
	for any $\tau<0$.
	Further, for any positive radial strictly monotone decreasing $C^2$-supersolution $v$ of \eqref{mainhardy-impr}  we have that 
	\begin{equation}\label{tHimpr2}\limsup_{r\to\infty}\frac{  |x|^\frac{p-N}{p}\log^{\bar\beta_{\epsilon} }|x|\log\log^\tau|x|  }{v(x)}>0,
	\end{equation}
	 if $\epsilon<C_*$ and  $\tau >0$ or  $\epsilon= C_*$ and  $\tau>2/p$.
\end{thm} 

\proof
First assume $\eps<C_*$. Fix $\tau<0$ and choose $\rho_0>1$ so that the subsolution
$$u(x)=|x|^\frac{p-N}{p}\log^{\beta_\eps}|x|\log\log^\tau|x|$$ 
is monotonically smaller in $B[\rho_0,\infty)$ than the supersolution defined by 
$$v(x)=|x|^\frac{p-N}{p}\log^{1/p}|x|$$ 
and such that both $u$ and $v$ are decreasing (see Table~\ref{table} and recall that $1/p\in (\beta_{\epsilon}, \bar\beta_{\epsilon})$). Note that function $v$ guarantees the validity of condition $(*)$. 
Thus Lemma~\ref{super} is applicable to the pair of functions $u,v$ in $B[\rho_0,\infty)$ and the conclusion follows.

If $\eps=C_*$ we have $\beta_{\eps}=\bar\beta_\eps=1/p$ and the previous argument is not applicable.
Instead, consider  
$$v=|x|^\frac{p-N}{p}\log^{\beta_\eps}|x|\log\log^{1/p}|x|,$$
and choose $\rho_0>0$ such that both $u$ and $v$ are monotone decreasing and $u$ is monotonically smaller than $v$ in $B[\rho_0,\infty)$.
Thus \eqref{tHimpr1} follows by Lemma~\ref{super}.\\
In order to prove \eqref{tHimpr2}  for $\epsilon <C_*$ we note that the subsolution  $u= |x|^\frac{p-N}{p}\log^{\bar\beta_{\epsilon} }|x|\log\log^\tau|x| $ with $\tau >0$  
and the supersolution $w= |x|^\frac{p-N}{p}\log^{\bar\beta_{\epsilon} }|x|\log^{\tilde\tau} \log  |x|$ with $\tilde\tau <0$  do not satisfy estimate \eqref{c-super1} hence alternative (i)
in Theorem~\ref{c-super} doesn't hold. Thus alternative (ii), hence \eqref{tHimpr2}   must necessarily hold.   The case $\epsilon =C_*$ can be treated in the same way using the subsolution $u= |x|^\frac{p-N}{p}\log^{1/p }|x|\log\log^\tau|x| $ with $\tau >2/p$ and the supersolution $v= |x|^\frac{p-N}{p}\log^{1/p}|x|\log\log^{\tilde\tau}|x| $ with $0<\tilde\tau <2/p$ (see Table~\ref{table} and recall that $\bar\beta_{\epsilon}=1/p$ in this case).
\qed

We  now apply Theorem~\ref{c-super-2}  to equation \eqref{mainhardy-impr} with $\epsilon \in ]0, C_*]$  in any exterior domain $B[r_0, \infty [$ with $r_0>1$.

\begin{thm} \label{tHbcc} Let $p\in [2,\infty [$ be such that $p>N$ and let $\epsilon \in ]0, C_*]$ and $r_0>e$.
	Then for any positive  radial strictly monotone increasing $C^2$ strict  supersolution $v$ of \eqref{mainhardy-impr}  we have
	\begin{equation}\label{tH1bcc}|x|^\frac{p-N}{p}\log^{\beta_\eps}|x|\log\log^\tau|x|  \lesssim v\quad\text{in $B[r_0, \infty [$},\end{equation}
	for any $\tau <0$.
	Further, for any positive (not necessarily radial or increasing) $C^2$-supersolution $w$ of \eqref{mainhardy-impr}   we have
	\begin{equation}\label{tH2bcc}\limsup_{r\to \infty}\frac{|x|^\frac{p-N}{p}\log^{\bar\beta_\eps}|x|\log\log^\tau|x|  }{w(x)}>0,
	\end{equation}
	  if $\epsilon<C_*$ and  $\tau >0$ or  $\epsilon= C_*$ and  $\tau>2/p$.
 \end{thm} 

\proof The proof is similar to the proof of Theorem~\ref{tHb}.    
For $\rho_0>0$ sufficiently large, consider in $B[\rho_0, \infty [$  the  subsolution defined by $u(x)=|x|^\frac{p-N}{p}\log^{\beta_\eps}|x|\log\log^\tau|x|$ with $\tau <0$, 
 and the supersolution defined by $w^*(x)= |x|^\frac{p-N}{p}\log^{\beta_\eps}|x|\log\log^{\tilde \tau}|x|$  with $0<\tilde \tau <2/p$ (note that this choice of $\tilde \tau$ allows to treat also the case $\epsilon=C_*$), see Table~\ref{table}. Note that function $w^*$ guarantees the validity of condition $(*)$.   Since inequality \eqref{principle1bis} is not satisfied by the couple of functions $u,w^*$ then alternative $(i)$ in Theorem~\ref{c-super-2} holds in  $B[\rho_0,\infty [$. Thus estimate \eqref{tH1bcc} is satisfied in $B[\rho_0,\infty[$, hence also in $B[r_0,\infty[$ for any $v$ as in statement.  
 
To establish \eqref{tH2bcc}, we argue in a similar way by taking the subsolution  defined by $\bar u (x)=|x|^\frac{p-N}{p}\log^{\bar\beta_\eps}|x|\log^\tau \log|x| $ with   $\tau >0$  if $\epsilon<C_*$ and  $\tau>2/p$  if $\epsilon= C_*$ and the strict  supersolution  $v^*(x)= |x|^\frac{p-N}{p}\log^{\bar\beta_\eps}|x|\log^{\tilde\tau}\log|x|$ with  $\tilde\tau <0$  if $\epsilon<C_*$ and  $0<\tilde \tau <2/p$  if $\epsilon= C_*$.     Since \eqref{limitesupbis} is not satisfied by the couple of functions $\bar u, v^*$ then alternative $(i)$ in Theorem~\ref{c-super-2} doesn't hold hence \eqref{tH2bcc} holds for any $w$ as in the statement. 
\hfill $\Box$

\begin{rem} If $p=N\geq 2$  the statement of Theorem~\ref{tHbcc} has to be modified as follows. Estimate \eqref{tH1bcc}  with $\tau =0$ holds for any $\epsilon \in ]0, C_*]$.     Indeed,  if $\epsilon <C_*$  one can repeat the argument above by using  the subsolution 
$u(x)=\log^{\beta_\eps}|x| $ and the supersolution  $w^*(x)=\log^{\beta }|x| $ with $\beta_{\epsilon}<\beta <\bar\beta_{\epsilon}$ and see that the couple $u,w^*$
do not satisfy  \eqref{principle1bis}, hence alternative $(i)$ in Theorem~\ref{c-super-2} holds. For $\epsilon =C_*$ we have $\beta_{\epsilon}=\bar\beta_{\epsilon}$ and another supersolution $w^*$ has to be used. In this case, one can verify by a direct computation that  the function
$w^*(x)=\log^{\beta_{\epsilon} }|x|\log^{\tau} \log |x| $ with  $\tau >0$  is a supersolution and the same argument can be repeated. 

Similarly, estimate \eqref{tH2bcc} with $\tau =0$ can be proved for any $\epsilon \in ]0,C_*[$. 
Indeed one can see that the couple given by the subsolution  $u(x)=\log^{\bar \beta_\eps}|x| $ and the strict supersolution $v^*(x)=\log^{\beta }|x| $ with $\beta_{\epsilon}<\beta <\bar\beta_{\epsilon}$ do not satisfy  \eqref{limitesupbis} hence  \eqref{tH2bcc} holds for any $w$ as in the statement.
\end{rem}

\begin{rem}
The results of Theorems \ref{tH},  \ref{tHb}, \ref{tHimpr} and  \ref{tHbcc}  stated in a slightly different form are essentially contained in \cite[Theorem 3.4 and 3.5]{LLM}, except of the critical case $\eps=C_*$, which is new. The proofs of the upper bounds in \cite[Theorem 3.5]{LLM} are based on an intricate construction of explicit barriers vanishing at the boundary that uses generalised Pr\"ufer transformation. We believe that Phragm\'{e}n-Lindel\"{o}f principle in this paper provides a more direct proof of upper bounds on supersolutions that could be adapted to different classes of domains and potentials, e.g. potentials that involve distances to the boundary.
\end{rem}

\noindent
{\bf Acknowledgments.}  The authors are greatful to Prof. Antonio Vitolo for references and useful discussions. 
P.D. Lamberti  is a  member of the Gruppo Nazionale per l'Analisi Matematica, la Probabilit\`a e le loro Applicazioni (GNAMPA) of the Istituto Nazionale di Alta Matematica (INdAM) and acknowledges partial support 
  of the ``INdAM GNAMPA Project'' codice CUP\_E53C22001930001 ``Operatori differenziali e integrali in geometria spettrale''.
Both authors acknowledge the warm hospitality received by each other institution on the  occasion of several research visits. 

\end{document}